\documentclass[12pt]{amsart} %%%{article}
\usepackage{amssymb}
\usepackage{amsmath}
\usepackage{stmaryrd}

\newtheorem{theorem}{Theorem}[section]

\newtheorem{lemma}[theorem]{Lemma}

\newtheorem{remark}[theorem]{Remark}

\numberwithin{equation}{section}

\providecommand{\keywords}[1]{\textbf{\textit{Index terms---}} #1}

\begin{document}

\title{Conformal Ricci flow on asymptotically hyperbolic  manifolds}

\author{Peng Lu, Jie Qing, and Yu Zheng}
\date{}

\begin{abstract}
In this article we study the short-time existence of conformal Ricci flow on asymptotically 
hyperbolic manifolds.
We also prove a local Shi's type curvature derivative estimate for conformal Ricci flow.  
\end{abstract}



\address{Peng Lu, Department of Mathematics, University of Oregon, Eugene, OR 97403, USA}
\email{penglu@uoregon.edu}
 
 \address{Jie Qing, Department of Mathematics, University of California at Santa Cruz,  
 Santa Cruz, CA 95064, USA}
 \email{qing@ucsc.edu}

\address{Yu Zheng, College of Mathematics, East China Normal University, Shanghai, P.R. China}
\email{zhyu@math.ecnu.edu.cn}

\thanks{P.L. is partially supported by Simons Foundation Collaboration Grant 229727,
J.Q. is partially supported by NSF DMS-1608782, and Y.Z. is partially supported by CNSF 11671141.}

\maketitle

\section{Introduction}  \label{sect introduction}

The geometry and analysis on asymptotically hyperbolic (AH in short) manifolds  
attracted significant research interest from both mathematics and theoretic physics 
communities, particularly after the introduction of AdS/CFT correspondence in theoretic 
physics (cf. \cite{Ma98, GKP, Wi98}). In this paper we prove the short time 
existence on AH manifolds and a local Shi's type curvature derivative
 estimates for conformal Ricci flow (CRF in short).

Ricci flow is known to be a powerful geometric and analytic tool in differential geometry 
and topology. CRF was introduced  by Fischer \cite{Fi04}
as the modified Ricci flow that maintains scalar curvature constant. It is so named because 
the constancy of scalar curvature is achieved by the conformal 
deformation of metrics at each time. Fischer \cite{Fi04} observed that, on compact 
manifolds, Yamabe constant is strictly increasing along CRF. Later, in \cite{LQZ},
the short-time existence of CRF on asymptotically flat 
manifolds was established. Interestingly, it is observed that ADM mass is strictly decreasing 
unless the initial metric is Ricci-flat \cite[Theorem 1.4]{LQZ} (in contrast to the fact that ADM 
mass stays constant along Ricci flow \cite{DM}). CRF is considered to possibly be an 
efficient way to search for Einstein metrics because of the nature of the Einstein-Hilbert
action (cf. \cite{Be87, Fi04}), that is, Einstein metrics on a manifold $M^n$ 
of dimension $n$ may be associated with
$$
\sup_{\{[g]: \text{conf classes}\}}\inf_{\{g\in[g]: \text{Riem metrics}\}}\frac{\int_{M}
 R_g d \mu_g }{\text{vol}(M, g)^\frac {n-2}n}.
$$

As introduced in \cite{Fi04} (see also \cite{LQZ}),  a family of metrics $\{g(t): t\in [0, T)\}$ 
on a smooth manifold $X^{m+1}$ is said to be CRF if it satisfies: 
\begin{equation} \label{eq crf g p orig general}
\left\{\aligned
& \partial_t g(t) + 2 \left ( \operatorname{Rc}_{g(t)} + m g(t) \right ) =  -2 p(t) g(t) \quad \text{in } 
X \times (0,T), \\
& (-\Delta_{g(t)} + (m+1))p(t)  = \frac 1m| \operatorname{Rc}_{g(t)} + m g(t) |_{g(t)}^2 
 \quad \text{on } X \times [0, T), \\
& g(0) = g_0,
\endaligned\right.
\end{equation}
where the initial metric $g_0$ has constant scalar curvature $- m(m+1)$, 
$p(t)$ is an auxiliary function and is named as the pressure function in \cite{Fi04} 
when comparing CRF with the Navier-Stokes equations, and $T$ is a positive constant. 
The scalar curvature $R_{g(t)}$  remains as $-m(m+1)$ for all $t\in [0, T)$. 
\\

Before stating the results, let us first briefly introduce AH manifolds. Let $\bar{X}$ be a compact smooth 
manifold with nonempty smooth boundary $\partial X$ and let $X$ be the interior. A smooth function 
$x: \bar{X} \rightarrow [0, \infty)$ is called a defining function for the boundary $\partial X$ if it satisfies:
$$
\text{1) $x > 0$ on $X$; \quad 2) $x = 0$ on $\partial X$; \quad 3) $dx \neq 0$ on $\partial X$}.
$$
A metric $g$ on $X$ is $C^{l+\beta}$ conformally compact if $x^2g$ 
extends to be a  $C^{l +\beta}$ metric on $\bar{X}$ for a boundary defining function $x$, 
where  $l \geq 2$ is an integer and $\beta \in(0,1)$.  $x^2g$ induces a metric $\hat g$ on the boundary and, 
in fact, $g$ induces a conformal structure $[\hat g]$ on the boundary when defining functions vary. 
$(X, g)$ is said to be AH if it is conformally compact and the sectional curvature of $g$ goes to $-1$ 
asymptotically at the infinity.

 It is often convenient  to use geodesic defining functions for AH manifolds. A geodesic defining 
 function $x$ is a defining function such  that  $|dx|_{x^2 g} = 1$ in a collar neighborhood of 
 the boundary not just on the boundary. Hence 
$$
g = x^{-2} (dx^2 + g_x)
$$
where $g_x$ is a family of metrics on $\partial X$ depending on $x$, i.e. the metric $g$ splits orthogonally 
in $x$ and the tangential to $\partial X$.

The central feature of AH manifolds is the association of the AH Riemannian 
metrics on $X$ to the conformal structures on its boundary $\partial X$, 
which is fundamental to a mathematical theory of AdS/CFT correspondence in theoretic physics. 
Similar to the explorations in \cite{Fi04, LQZ}, one expects 
that CRF is significant in the search for appropriately canonical AH metrics 
for a given conformal structure at the 
boundary to enrich the mathematical theory for AdS/CFT correspondence.  
\\

The first result of this article is the short time existence of CRF on AH manifolds (see \S 
\ref{subsec notat func space} for the description of H\"{o}lder spaces used).

\begin{theorem}\label{thm short exist CRF on AH}
Let $(X^{m+1}, g_0)$, $m\geq 3$, be a  $C^{4+ \alpha}$ AH manifold with constant scalar curvature 
$-m(m+1)$ and let $x$ be a geodesic defining function. Assume that 
$\operatorname{Rc}_{g_0} + m g_0\in x^2 C^{2+\alpha}_e(X)$. Then, for some $T > 0$,  
there is a family of metrics $g(t) = g_0 + u(\cdot, t)$ which solves CRF \eqref{eq crf g p orig general} 
such that $g(t)$ is $C^{1+\alpha}$ 
AH with constant scalar curvature $-m(m+1)$ 
and $u \in x^2C^{2+\alpha, \frac {2+\alpha}2}_e(X_T)$.
\end{theorem}

We would like to note that one can always conformally deform a given AH metric into an AH metric
with constant scalar curvature, say $-m(m+1)$, thanks to \cite{ACF}. 
Let $\owedge$ be the Kulkarni-Nomizu product.
The conditions that  Riemannian curvature 
$\operatorname{Rm} + g\owedge g = O(x^2)$, Ricci curvature $\operatorname{Rc}+ m g= O(x^2)$, 
and the boundary $\partial X$ is totally geodesic in $X$ under the metric $x^2 g$ for 
a geodesic defining function $x$, are all equivalent 
and preserved under the conformal deformations taken in \cite{ACF} (cf. Lemma \ref{calculation-AH-curvature}). 
And the property that $\operatorname{Rm} + g\owedge g = O(x^2)$ on an AH 
 manifold is intrinsic and independent of the choice of the geodesic defining functions.

Ricci flows on complete noncompact manifolds were studied by many people. The most notable early
work is \cite{Sh89} by Shi in 1989. Ricci flows on AH manifolds were also studied in \cite{QSW, Ba11},
where the existence in \cite{QSW} was based on the maximum principle argument and the 
existence result in \cite{Sh89}, while the existence in \cite{Ba11} is based on asymptotic analysis on
AH manifolds in \cite{Mz91, Le06, Al07}. In this paper, we prove Theorem \ref{thm short exist CRF on AH} 
using the framework similar to that in \cite{LQZ} 
for parabolic-elliptic systems  based on the analysis on AH manifolds from 
\cite{Mz91, Le06, Al07, Ba11}. 
\\

The second result is a local Shi's type estimate for CRF. In Ricci flow Shi's estimates
 on derivatives of Riemannian curvature is crucial for 
 compactness results, and therefore, are essential for the 
later developments in Ricci flow.

 For CRF on smooth manifold $M^{n}$ with initial metric $g_0$ of constant scalar curvature 
  $R_{g_0}=2nc$

\begin{equation} \label{eq crf g p local riem}
\left\{\aligned
& \partial_t g(t) = -2 \left ( \operatorname{Rc}_{g(t)} - 2c
 g(t) \right ) -2 p(t) g(t) \quad \text{on } M \times (0,T], \\
& ((n-1) \Delta_{g(t)} + 2nc ) p(t)  = - | \operatorname{Rc}_{g(t)}
-2cg(t) |_{g(t)}^2  \quad \text{on } M \times [0, T],
\endaligned\right.
\end{equation}
 we have

\begin{theorem}\label{thm local curvature derivative est}
Fix constants $\alpha  \in (0, 1),  K \geq 1, \tilde{K}>0, c,  r>0$, and integer $n \geq 2$, 
we have the following.

\vskip .1cm
\noindent $($i$)$ There exists a constant $C_1=C_1\left( \alpha, n, \sqrt{K}r,\tilde{K}, |c| \right) $
depending only on $\alpha, n$, $\sqrt{K}r, \tilde{K} $, and $|c|$,
such that the following property holds.
Let $(M^n, g(t), p(t)), \, t\in\left[
0,T \right] $, be a solution to the CRF (\ref{eq crf g p local riem}).
Assume that closed ball $\bar{B}_{g\left( 0\right)
}\left( O,r\right) \subset M $ is compact and that
\begin{align}
&\left\vert \operatorname{Rm}  \right\vert \leq K \quad \text{ on } \bar{B}_{g\left(
0\right) }\left( O,r\right) \times \lbrack0, T_*],
\label{eq curv bdd assump}  \\
& \max_{i=0,1,2,3} | \nabla^i p| \leq \tilde{K} \quad \text{ on } \bar{B}_{g
\left(  0\right) }\left( O,r\right) \times \lbrack 0, T_*],   \label{eq p deriv bdd assump}
\end{align}
where constant $ T_* \leq \min \{T, \alpha /K \}$,  then we have
\begin{equation}
\left\vert \nabla\operatorname*{Rm}\left(  x,t\right)  \right\vert_{g(t)} \leq
\frac{C_1 K}{\sqrt{t}} \label{eq est local first order Shi}
\end{equation}
for all $\left( x,t\right)  \in B_{g\left(  0\right)  }\left(
O,r/2\right)  \times(0,T_*]$.

\vskip .1cm
\noindent $($ii$)$ If  $(M^n, g(t),p(t)), \, t \in [0,T]$, in (i) is a complete solution to the CRF. 
Suppose assumptions
(\ref{eq curv bdd assump}) and (\ref{eq p deriv bdd assump}) hold on $M \times [0,T_*]$,
then there is a constant $C_2= C_2(\alpha, n,\tilde{K}, |c|)$ such that
\begin{equation}
\left\vert \nabla\operatorname*{Rm}\left( x,t\right)  \right\vert_{g(t)} \leq
\frac{C_2 K}{\sqrt{t}} \label{eq est first order Shi}
\end{equation}
for all $\left( x,t\right)  \in M \times(0,T_*]$.
\end{theorem}

Here is the outline of the rest of this article.
In \S \ref{sect preli map proper} we discuss basic analysis results on AH manifolds from 
\cite{Le06, Ba11}, which are needed to prove the short time existence. 
\S \ref{sect existence of CRF on AH mfld} is devoted to the proof of Theorem 
\ref{thm short exist CRF on AH} using Banach's contraction mapping theorem.
In \S \ref{sect Shi's derivative estimate under crf} we give a proof Theorem 
\ref{thm local curvature derivative est} using the parabolic maximum principle.
Also we will outline a proof of high order derivative estimate of the curvature tensor
of CRF (see Theorem \ref{thm shi high order estimate}).
\\

\noindent \textbf{Acknowledgement}. 
P.L. thanks the visiting program for scholars from abroad at Peking University and
Professor Zhu, Xiaohua for the warm hospitality and support. P.L. and J.Q. thanks Beijing International 
Center for Mathematical Research, where part of this work is carried out during the summer of 2017.

%%%%%%%%%%%%%%%%%%%%%%%%%%%%%%%%%%%%%%%%%%%%%%%%%%
%%%%%%%%%%%%%%%%%%%%%%%%%%%%%%%%%%%%%%%%%%%%%%%%%%
\section{Preliminaries for AH manifolds} \label{sect preli map proper}

In this section we recall some  basic analysis on AH manifolds. We mostly rely on 
\cite{Mz91, Le06, Al07, Ba11},
and readers are referred to them for detailed accounts.

%%%%%%%%%%%%%%%%%%%%%%%%%%%%%%%%%%%%%%%%
\subsection{Notations and function/tensor H\"{o}lder spaces}\label{subsec notat func space}
Let $({X}^{m+1}, g)$ be a  $C^{l+ \beta}$ AH manifold with $l \geq 2$ and $\beta\in (0, 1)$.  
Let $x$ be a geodesic defining function such that, in a collar neighborhood 
$\{ x < \varepsilon\}$ near the infinity,  we have $g = x^{-2}(dx^2 + g_x)$,
where $g_x$ is a family of metrics on $\partial X$. Let $X_T = X\times [0, T]$. 
We will consider the function/tensor H\"{o}lder spaces 
$$
x^\mu C^{k+\alpha}_e(X), \quad x^\mu C_e^{k+\alpha, \frac{k+\alpha}2}(X_T)
$$ 
as described in \cite[Section 3.1]{Ba11}, for $k + \alpha \leq l+\beta$, 
which is assumed throughout this paper.

We are concerned with the operators $-\Delta+(m+1)$ on functions and 
$\Delta _L + 2m$ on 2-tensors on $X$,  where the  Lichnerowicz operator $\Delta_L =\Delta_L^g $ 
acting on $2$-tensor $w$ is given by
\[
(\Delta_L^{g} w)_{ij} = g^{kl} \nabla_k \nabla_l  w_{ij} + 2 g^{kl}g^{pq} 
 R_{ikpj}w_{lq} - g^{kl} R_{ik}w_{lj} - g^{kl}R_{jk}w_{li}.
\] 
Throughout this article when we use local coordinates $(x^0, x^1, \cdots, x^m)$ 
near the boundary of $X$,  we choose  $x^0=x$ and $(x^1, \cdots, x^m)$ to be
some local coordinates on $\partial X$.
 Our convention for Riemann curvature tensor components $R_{ijkl}$ is such that 
 $g^{kl} R_{iklj} = R_{ij}$.
\\

Let $\operatorname{Rm}$ denote the curvature $(4,0)$-tensor and 
let $\operatorname{Rc}$ denote the Ricci curvature.
First by some straightforward calculation we get

\begin{lemma} \label{calculation-AH-curvature} 
Suppose that $(X^{m+1}, g)$ is $C^{l+\beta}$ 
AH and that $x$ is a geodesic defining function. 
Let $\bar g = x^2 g$. 
Then near the boundary of $X$ we may write Riemannian curvature tensor
components as
\begin{equation}\label{AH-curvature}
R_{ijkl}(g) =  - (g_{ik}g_{jl} - g_{il} g_{jk})  + x^{-3} T_{ijkl} + x^{-2} R_{ijkl}(\bar g)
\end{equation}
for some $(4,0)$-tensor $T$ defined in (\ref{AH-curvature tensor T}) below. 
It follows that $\operatorname{Rm} + g \owedge g \in x^2 C^{l-2+
\beta}_e(X)$, $\operatorname{Rc} + mg \in x^2 C^{l-2+\beta}_e(X)$,  and
the condition that boundary $\partial X$ is totally geodesic in $X$ under metric $\bar{g}$,
 are all equivalent.  Note that the last condition is independent of the choice of the 
geodesic defining function $x$.
\end{lemma}
\proof Note that $\operatorname{Rm} + g \owedge g \in x^2 C^{l-2+
\beta}_e(X)$ is equivalent to that $T$ vanishes at the boundary $\partial X$,
also note that the condition that boundary $\partial X$ is totally geodesic in $X$ under $\bar{g}$
is equivalent to $\partial_x \bar g_{ab}|_{x=0} = 0$, where $a, b \in \{1, 2, \cdots, m\}$.

After a long but simple calculation, we get
\begin{equation}\label{AH-curvature tensor T}
2T_{ijkl} = \bar g_{ik}\partial_x\bar g_{jl} + \bar g_{jl}\partial_x\bar g_{ik} 
- \bar g_{il}\partial_x \bar g_{jk} - \bar g_{jk}\partial_x \bar g_{il}.
\end{equation}
Therefore $T$ vanishes at the boundary $\partial X$ if and only if $\partial_x \bar g_{ab}|_{
x=0} = 0$.

Taking the trace of (\ref{AH-curvature}) we have
\begin{equation} \label{AH-curvature Ricci}
R_{ij}(g)=-mg_{ij} +\frac{1}{2} x^{-1} \left (\bar{g}_{ij} \bar{g}^{kl} \partial_x \bar{g}_{kl}
+(m-1) \partial_x \bar{g}_{ij} \right ) -x^2R_{ij}(\bar{g}). 
\end{equation}
Hence  $\operatorname{Rc} + mg \in x^2 C^{l-2+\beta}_e(X)$ is equivalent to  
$\partial_x \bar g_{ab}|_{x=0} = 0$. The lemma follows from the equivalences proved above.
\endproof

Let $(X^{m+1}, g_0)$ be a $C^{l+\beta}$ AH manifold.
As a consequence of Lemma \ref{calculation-AH-curvature}, for time-dependent cases, we have
 \begin{equation} \label{eq ut Rm asymp behav}
\operatorname{Rm}_{g_u(t)}  +  g_u(t) \owedge g_u(t) \in  x^2 C_e^{k-2+\alpha, 
\frac{k -2+\alpha}2}(X_T)
 \end{equation}
for $g_u(t) = g_0 + u(\cdot, t)$ and a symmetric $2$-tensor $u \in x^2 C^{k+\alpha, \frac{k+\alpha}
2}_e(X_T)$, provided that $\operatorname{Rc}_{g_0} + m g_0 \in x^2 C^{l-2
+\beta}_e(X)$.  
We also obtain by a direct calculation 

\begin{align}
\|\operatorname{Rm}_{g_u(t)} -\operatorname{Rm}_{ g_{\tilde{u}} (t)} \|_{  x^{\mu} C_e^{k-2+\alpha,
\frac{k -2+ \alpha}2}(X_T)} \leq C \| u -\tilde{u} \|_{ x^{\mu}C^{k+\alpha, \frac{k+\alpha}2}_e(X_T)}
\label{eq u tilde u cur diff est}
\end{align}
for symmetric $2$-tensors $u, \tilde u \in x^\mu C^{k+\alpha, \frac {k+\alpha}2}_e (X_T)$. 
The other very useful fact for us is the following:

\begin{lemma} \label{Lemma 3.7-Lee} (\cite[Lemma 3.7]{Le06}) Let $(X^{m+1},  g)$ 
be a $C^{l+\beta}$ AH manifold. 
For 2-tensors, we have continuous embedding 
\begin{align}
x^\mu C^{k+\alpha}(\overline{X}) & \hookrightarrow x^{\mu+2} C^{k+\alpha}(X)
\label{Lemma 3.7-Lee-a}\\
x^{k+\alpha+2}C^{k+\alpha}(X) & \hookrightarrow C^{k+\alpha}(\overline{X}) 
\label{Lemma 3.7-Lee-b}
\end{align}
for $k+\alpha \leq l+\beta$.
\end{lemma}

Consequently, we have that $g_u(t)$ is of $C^{k+\alpha}$ AH if $g_0$ is $C^{k+\alpha}$ 
AH and $u\in x^\mu C^{k+\alpha, 
\frac{k+\alpha}2}(X_T)$ for $\mu \geq k+\alpha$.

%%%%%%%%%%%%%%%%%%%%%%%%%%%%%%%%%%%%%%%%%%%%%%%%
\subsection{Elliptic Schauder estimates on AH manifolds}

To treat the pressure function $p$ when solving equation \eqref{eq crf g p orig general}, 
we recall the isomorphism property of $- \Delta +(m+1)$ 
on spaces of functions.
The following is an immediate consequence of \cite[Lemma 3.3]{Le95}.

\begin{lemma} \label{isomorphism-function} (\cite[Lemma 3.3]{Le95})
Let $(X^{m+1}, \ g)$ be a $C^{l+\beta}$ AH manifold. Then, for $k-1+\alpha \leq l
+\beta$ and $\mu\in (-1, m+1)$, 
$$
-\Delta_g + (m+1): x^\mu C^{k+\alpha}_e(X) \to  x^\mu C^{k -2 +\alpha}_e (X)
$$
is an isomorphism. 
\end{lemma}

For our purpose we need to solve the pressure equation $p(\cdot, t)$ in 
\eqref{eq crf g p orig general} for each $t\in [0, T]$, in other words,  
we need Schauder estimates uniform in the time variable like [LQZ, Lemma 3.11]. 
To apply Lemma \ref{isomorphism-function} to metric $g_u(t) = g_0 + u (\cdot, t)$ 
at each $t \in [0, T]$,
we need $g_u(t)$ to be at least $C^{k-1+\alpha}$ AH and 
close to $g_0$ in some appropriate sense. In fact, as a consequence
of Lemma \ref{Lemma 3.7-Lee} and Lemma \ref{isomorphism-function}, we have 

\begin{lemma} \label{Lemma 3.11-LQZ}
Suppose that $(X^{m+1}, g_0)$ is $C^{l+\beta}$ AH. Then for $k-1+\alpha\leq l+\beta$
and $\mu\in (-1, m+1)$, 
there exist $\delta >0$ and $C>0$ such that
\begin{equation} \label{schauder-t}
\|(-\Delta_{g_u(t)} + (m+1))^{-1} \phi \|_{x^\mu C^{k+\alpha, \frac{k-2+\alpha}2}_e(X_T)}\leq C 
\|\phi\|_{x^\mu C^{k-2+\alpha, \frac{k-2+\alpha}2}_e(X_T)},
\end{equation}
provided that
$$\|u\|_{x^\nu C^{k-1+\alpha, \frac{k-1+\alpha}2}_e(X_T)} \leq \delta$$ and $\nu \geq k-1+\alpha$.
\end{lemma}

For a general initial metric $g_0$ with $C^{l+\beta}$ AH regularity, according to \cite{Ba11}, 
one may expect to work with $u \in xC^{2+\alpha, \frac {2+\alpha}2}_e(X_T)$
with $2+\alpha \leq l+\beta$ to prove the short time existence of solutions of form $g_u(t)$
by applying the contraction mapping theorem. 
This is due to, initially, 
$$
\operatorname{Rc}_{g_0} + m g_0 \in xC^{l-2+\beta}_e(X)
$$
(cf. \eqref{AH-curvature Ricci}). In the light of \eqref{Lemma 3.7-Lee-b}, 
 this only provides 
$C^\alpha$ AH regularity for the metrics $g_u(t)$ after initial time, 
which is not enough. 
We need the minimal $C^{1+\alpha}$ AH regularity for $g_u(t)$ in order to 
apply Lemma \ref{Lemma 3.11-LQZ}, and Lemma  
\ref{lem Laplace invertible weighted timed} and \ref{thm 3.2 in Ba11} below. 
Therefore one needs, at least for this technical reason, assume that
for the initial metric $g_0$
\begin{equation}\label{initial-condition}
\operatorname{Rc}_{g_0} + m g_0 \in x^2 C^{l-2+\beta}_e(X),
\end{equation}
i.e., $\partial X$ is total geodesic in $(\bar{X}, x^2g)$.
Later we will choose $k = 2$ and $\nu = 2 > 1+\alpha$ in applying Lemma \ref{Lemma 3.11-LQZ}. 
In the proof of Lemma \ref{lem Ba 4.5}, we will need $\|\operatorname{Rc}_{g_0} + m g_0
\|_{x^2 C^{2+\alpha}_e(X)}$ to be bounded, this leads us to choose $l =4$ and $\beta=\alpha$
in the proof of short time existence.
\\

Based on Lemma \ref{isomorphism-function} and Lemma \ref{Lemma 3.11-LQZ}, 
we may define 
\begin{equation} \label{eq cal P def}
\mathcal{P} (g) = \frac 1m (-\Delta_{g} + (m+1))^{-1}(|\text{Rc}_{g}+ mg|^2),
\end{equation}
for any $C^{1+\alpha}$ AH metric $g$.
Then we can easily derive the following.

\begin{lemma} \label{lem Laplace invertible weighted timed}
Let $(X^{m+1}, \ g_0)$ be a  $C^{l+ \beta}$ AH manifold. Let  $k \in \mathbb{N}, \alpha\in (0, 1)$, 
and $k + \alpha \leq l+\beta$. Then there are small positive constants $T$ and $\epsilon$ 
such that the following hold.
Let $g_u(t) = g_0+u(\cdot, t)$ with 
\[
u\in x^\mu C^{k+\alpha, \frac{\alpha}2}_e(X_T) \bigcap 
x^\nu C^{k-1+\alpha, \frac {\alpha}2}_e(X_T).
\]

\vskip .1cm
\noindent (i) For $ \|u\|_{x^\nu C^{k-1+\alpha, \frac{\alpha}2}_e(X_T)} \leq \epsilon$, 
we have $\mathcal{P}(g_u) \in x^{2\mu} C_e^{k+\alpha, \frac{\alpha}2}(X_T)$, provided 
that $2\mu\in (0, 2)$, $\nu > k-1+\alpha$;  and

\vskip .1cm
\noindent (ii) For $u, \tilde{u}$ in  some given ball in $x^\mu C^{k+\alpha, 
\frac{\alpha}2}_{e}(X_T)$ which satisfy 
\[
\| u \|_{x^\nu C^{k-1+\alpha, \frac{\alpha}2}_e(X_T)} < \epsilon, 
\text{ and } \| \tilde u \|_{x^\nu C^{k-1+\alpha, \frac{\alpha}2}(X_T)} < \epsilon,
\] 
 we have
 \[
\| \mathcal{P}(g_u) -  \mathcal{P}(g_{\tilde{u}}) \|_{x^{\mu} C^{k+\alpha, \frac{\alpha}2}_{e}(X_T)} 
 \leq C \|u - \tilde u \|_{x^\mu C^{k+\alpha, \frac {\alpha}2}_e(X_T)},
 \]
 provided that $\mu \in (0,  m+1)$, $\nu > k-1+\alpha$.
 \end{lemma}
 
We will use Lemma \ref{lem Laplace invertible weighted timed} with $\nu = 2$ 
and $k = 2$  below.

%%%%%%%%%%%%%%%%%%%%%%%%%%%%%%%%%%%%%%%
\subsection{Parabolic Schauder estimates on AH manifolds}

We will need the following basic parabolic Schauder estimate \cite[Theorem 3.2] {Ba11}, 
which covers the case
when $L$ is $\Delta_L + 2m$ on 2-tensors.

\begin{lemma}\label{thm 3.2 in Ba11} (\cite[Theorem 3.2]{Ba11}) 
Suppose that $(X^{m+1}, g_0)$ is $C^{l+\beta}$ AH. 
Suppose $L$ is a second-order linear uniformly degenerate elliptic operator 
with time-independent coefficients. Let $k+\alpha\leq l+\beta$.
Then for every $f \in x^{\mu} C^{k+\alpha, \frac{k+\alpha}2}_e (X_T)$ 
there is a solution $v$ in $x^{\mu} C^{k+2+\alpha, 
\frac{k+2+\alpha}2}_e (X_T )$ to equation
\begin{equation}\label{linear-parabolic}
(\partial_t - L) v(x,t) =f(x,t) \quad \text{ and } v(x,0) =0.
\end{equation}
Moreover, $v$ satisfies the parabolic Schauder estimate
\[
\| v \|_{x^{\mu} C^{k+2+\alpha, \frac{k+2+\alpha}2}_e (X_T)} \leq C \| f \|_{ 
x^{\mu} C^{k+\alpha, \frac{k+\alpha}2}_e (X_T )},
\]
where constant $C=C(T)$ is bounded when $T$ is small.
\end{lemma}

%%%%%%%%%%%%%%%%%%%%%%%%%%%%%%%%%%%%%%%%%%%%%%%%%%
\subsection{AH metrics of constant scalar curvature} 
\label{subsec Holder space for CRF AH mflds}

In this subsection we recall the existence of a unique conformal deformation on a given AH manifold
to make the scalar curvature constant due to \cite[Theorem 1.2 and 1.3]{ACF}. 
This is significant because such a 
conformal deformation does not alter the conformal infinity of the AH manifold. 

\begin{lemma} \label{Theorem 1.2-ACF} 
Suppose that $(X^{m+1}, \ g)$ is smooth AH. Then there exists a unique 
conformal deformation $w^{\frac 4{m-1}}g$ such that
\begin{itemize}
\item $w$ is positive and in $C^{l+\beta}(\overline{X})$ for $l+\beta < m+1$;
\item metric $w^{\frac4{m-1}}g$ is $C^{l+\beta}$ AH with constant scalar curvature $-m(m+1)$;
\item $w(p) \to 1$ when $p$ approaches $\partial X$.
\end{itemize}
\end{lemma}

In the light of Lemma \ref{calculation-AH-curvature} we have

\begin{lemma}\label{Ricci-0-kept} Let $(X^{m+1}, g )$ be a smooth AH manifold. 
Suppose  that $\operatorname{Rc}_g + m g \in x^2C^{l-2+\beta}_e(X)$ with $l+\beta < m+1$. 
Then the conformal metric $w^\frac 4{m-1} g$  given in Lemma \ref{Theorem 1.2-ACF} 
is $C^{l+\beta}$ AH, and its traceless Ricci curvature is in 
$x^2C^{l-2+\beta}_e(X)$.
 \end{lemma}

%%%%%%%%%%%%%%%%%%%%%%%%%%%%%%%%%%%%%%%%%%%%%%%%%%
%%%%%%%%%%%%%%%%%%%%%%%%%%%%%%%%%%%%%%%%%%%%%%%%%%
\section{Short time existence for CFR on AH manifolds} \label{sect existence of CRF on AH mfld}

Based on our discussion of (\ref{initial-condition}) and the discussion of \cite{ACF} in 
\S \ref{subsec Holder space for CRF AH mflds},  from now on, 
we assume that the initial metric $g_0$ is $C^{4+\alpha}$ AH with constant scalar curvature 
$-m(m+1)$ and $\operatorname{Rc}_{g_0} + m g_0\in x^2C^{2+ \alpha}_e(X)$.

In this section we give a proof of Theorem \ref{thm short exist CRF on AH} in three steps.
We will first introduce and solve for a short time DCRF (short for DeTurck conformal Ricci flow). 
Then we convert the solution of DCRF to a solution of CRF.

 %%%%%%%%%%%%%%%%%%%%%%%

\subsection{DeTurk CRF and linearization}\label{subseq decomp evol op Ba11 trick}
Now we prepare for the application of the contraction mapping theorem to prove 
the short time existence of DCRF. Let $h_0 =g_0$.
The DCRF corresponding to CRF (\ref{eq crf g p orig general}) on a $C^{4+\alpha}$ 
AH manifold $(X^{m+1}, h_0)$ is 
\begin{equation}\label{eq dcrf}
\left\{\aligned 
& \partial_t h(t)  = -  2 \left (\operatorname{Rc}_{h(t)}+ m h(t) \right ) + \mathcal {L}_{W(t)} h(t)
 -2 \pi(t) h(t) \quad \text{on } X_T, \\
&  (\Delta_{h(t)}  -(m+1))\pi(t)  = - \frac{1}{m} |\operatorname{Rc}_{h(t)}+ m h(t) |_{h(t)}^2 
\quad \text{on } X_T, \\
& h(0) =h_0
\endaligned \right. 
\end{equation}
where vector field $W(t)$ is defined by 
\begin{equation}\label{vector field}
W^k(t) := h_0^{ij}  \left ( \Gamma_{ij}^k(h(t)) - \Gamma_{ij}^k (h_0)  \right ),
\end{equation}
$ \mathcal {L}_{W(t)}$ is the Lie derivative, and $\Gamma_{ij}^k$ is the Christoffel 
symbol of the corresponding metric.

For simplicity, in this section we use $\nabla$,  $\operatorname{Rm}$, and $\operatorname{Rc}$
 to denote the Levi-Civita
connection, Riemann curvature, and Ricci curvature of the metric $h_0$ respectively.
From \cite[(4.2)]{Ba11}  and \cite[Lemma 4.1, 4.2, 4.3]{Ba11} we have the following decomposition.
Let $h(t) = h_0 + u(\cdot, t) $ and let operator $L = \Delta_L^{h_0}+ 2m$, 
where $\Delta_L^{h_0}$ is the Lichnerowicz operator 
of $h_0$. We rewrite
\begin{equation}
 -  2 \left (\operatorname{Rc}_h+ m h \right )+ \mathcal {L}_{W} h = Lu+\mathcal{Q}(u) +\mathcal{E},
\label{eq Ba11 decomp of evol op}
\end{equation}
where operator
\begin{align*}
(\mathcal{Q} (u))_{ij} = & ((h_0 +u)^{kl}- h_0^{kl} ) {\nabla}_k {\nabla}_l u_{ij} 
+ (h_0+u)^{kl}(h_0+u)_{ip} (h_0) ^{pq}  {R}_{jklq} \\
& + (h_0+ u)^{kl} (h_0 +u)_{jp} (h_0)^{pq} {R}_{iklq} - 2 {R}_{ij} 
-2 {R}_{iklj}u_{kl} \\
& - {R}_{ik}u_{kj} - {R}_{jk}u_{ki} + (h_0 +u)^{-1} * (h_0+ u)^{-1}
* {\nabla} u * {\nabla}u, \\
 \mathcal{E} = &  2(\operatorname{Rc}_{h_0}+m h_0) \in x^2C^{2+ \alpha}_e(X).
\end{align*}
Note that $\mathcal{E}$ only depends of the initial metric $h_0$. 
The following is from  \cite[(4.4) and (4.5)]{Ba11}.
 
\begin{lemma}\label{Lem 4etc in Ba11}
We have
 \[
\mathcal{Q}(\cdot): x^{2} C_e^{2+\alpha, \frac\alpha 2}(X_T) \rightarrow x^{2}
 C_e^{\alpha, \frac \alpha 2}(X_T).
 \]
 Moreover
 \begin{align*}
&  \|\mathcal{Q}(u)\|_{x^{2}C^{\alpha, \frac {\alpha}2}_e(X_T)} \leq C \|u \|^2_
 {x^2 C^{2+\alpha, \frac{\alpha}2}_e(X_T)}, \quad \text{ and}  \\
  &  \| \mathcal{Q}(u) - \mathcal{Q}( \tilde{u}) \|_{ x^{2} C_e^{\alpha, \alpha /2}(X_T) }
   \leq C \max \{ \|u \|_{ x^{2}  C_e^{2+\alpha,  \alpha /2}(X_T) },  \|  \tilde{u}\|_{ x^{2} 
    C_e^{2+ \alpha, \alpha /2}(X_T)  } \} \\
 & ~\hskip 5.5cm \cdot \|u -  \tilde{u} \|_{ x^{2}  C_e^{2+\alpha,  \alpha /2}(X_T) }.
 \end{align*}
 \end{lemma}
 
\vskip .1cm
Next we want to introduce the analogous decomposition for the pressure function $\pi(t)$
in (\ref{eq dcrf}). 
Let $\mathcal{P}$ be the operator defined in (\ref{eq cal P def}) and let
\begin{equation}
-2 \pi(t) h(t) = -2\mathcal{P}(h_0+u) \cdot  (h_0+u) 
=-2 \mathcal{P}(h_0)u+ \hat{\mathcal{Q}} (u) 
+ \hat{\mathcal{M}}(u) + \hat{\mathcal{E}}, \label{eq P(h +u) decomp}
\end{equation}
where
$$
\aligned
\hat{\mathcal{E}} & = -2 \mathcal{P}(h_0)h_0, \\
\hat{\mathcal{M}} (u)  & = -2 (\mathcal{P}(h_0+u)-\mathcal{P}(h_0) ) h_0,  \\
\hat{\mathcal{Q}}(u) & =  -2(\mathcal{P}(h_0+u)-\mathcal{P}(h_0) )u.
\\
\endaligned
$$ 

Note that $\hat{\mathcal{E}}$ depends only on the initial metric $h_0$.  
Based on Lemma \ref{isomorphism-function} we have the following property of 
$\hat {\mathcal{E}}$ which is analog to that of $\mathcal{E}$.

 \begin{lemma}\label{lem funct prop Hat E(u)} Suppose that $\operatorname{Rc}_{h_0} 
 + m h_0 \in x^2 C^{2+ \alpha}_e(X)$. 
 Then we have $\hat{\mathcal{E}} \in x^{4} C_e^{4+\alpha}(X)$ and 
 there is a constant $C >0$ such that 
 \[ 
 \|\hat{\mathcal{E}} \|_{ x^{4} C_e^{4+\alpha}(X)} \leq C \|\operatorname{Rc}_{h_0} 
 + m h_0 \|^2_{x^2 C^{2+ \alpha}_e(X)}.
 \]
\end{lemma}

Based on Lemma \ref{lem Laplace invertible weighted timed}
 we have the following properties of $\hat{\mathcal{M}} (u)$, and properties
 of  $\hat {\mathcal{Q}}(u)$  which are analog to that of $\mathcal{Q}(u)$.

 \begin{lemma}\label{lem funct prop hat Q(u)} Suppose that 
 $\operatorname{Rc}_{h_0} + m h_0 \in x^2 C^{2+\alpha}_e (X)$ and that
 $u$ and $\tilde{u}$  are in some small balls centered at $0$ 
 in $x^2 C_e^{2+\alpha, \frac{\alpha}2}(X_T)$.
Then we have
 
 \vskip .1cm
\noindent $($i$)$  $\hat{\mathcal{M}}(u)  \in x^2 C_e^{2+\alpha, \frac{\alpha}2}(X_T)$
and there is a constant $C_1 >0$ such that 
\[
\| \hat{\mathcal{M}}(u)  \|_{ x^{2} C_e^{2+ \alpha,  \frac{\alpha}2}(X_T)} 
\leq C_1 \| u \|_{x^{2} C_e^{2+\alpha, \frac {\alpha}2}(X_T)};
\]

 \vskip .1cm
\noindent $($ii$)$  $\hat{\mathcal{Q}}(u)  \in x^2 C_e^{2+\alpha, \frac{\alpha}2}(X_T)$
and there is a constant $C_2 >0$ such that 
\[
\| \hat{\mathcal{Q}}(u)  \|_{ x^{2} C_e^{2 + \alpha,  \frac{\alpha}2}(X_T)} 
\leq C_2 \| u \|^2_{x^{2} C_e^{2+\alpha, \frac {\alpha}2}(X_T)};
\]

 \vskip .1cm
\noindent $($iii$)$ There is a constant $C_3 >0$ such that 
\begin{align*}
 &  \| \hat{\mathcal{Q}}(u) - \hat{\mathcal{Q}}(\tilde{u}) \|_{ x^{2} C_e^{2+\alpha,
  \frac{\alpha}2}(X_T) } 
 \leq C_3 \max \{ \| u \|_{ x^{2}  C_e^{2+\alpha,  \frac{\alpha}2}(X_T) },  \| \tilde{u} \|_{ x^{2} 
  C_e^{2+\alpha, \frac{\alpha}2}(X_T)  } \} \\
 & ~\hskip 5.5cm \cdot \| u- \tilde{u} \|_{ x^{2}  C_e^{2+\alpha,  \frac{\alpha}2}(X_T) }.
 \end{align*}
  \end{lemma}

%%%%%%%%%%%%%%%%%%%%%%%%%%%%%%%%%%%%%%%%%%%%%%%%
\subsection{Contraction mapping theorem for DCRF}\label{subsec DRF weak solvab}

In this subsection we will first prove the short time existence of solutions for DCRF that are 
$C^{1+\alpha}$ AH. 
 The Banach space we consider for the contraction mapping theorem is 
$x^2C^{2+\alpha, \frac {\alpha}2}_e(X_T)$.
For two positive parameters $\epsilon$ and $T$ to be specified below, 
we define a closed subset $Z_{\epsilon,T}$ as
\[
Z_{\epsilon,T} = \{ u \in  x^2 C^{2+\alpha, \frac \alpha 2}_e(X_T): \,  u(x, 0) = 0 \text{ and } 
\| u \|_{x^2 C^{2+\alpha, \frac {\alpha}2}_e (X_T )} \leq \epsilon \}.
\]

Let $h(t) =h_0 +  v(\cdot, t)$ in (\ref{eq dcrf}).
In the light of \eqref{eq Ba11 decomp of evol op} and \eqref{eq P(h +u) decomp}, 
we rewrite the DCRF as
$$
\left\{\aligned
(\partial_t - \hat L)v & = \mathcal{E}  + \mathcal{Q}(v) + \hat{\mathcal{E}} + 
\hat{\mathcal{Q}}(v) +\hat{\mathcal{M}}(v) \\
v(\cdot, 0) & = 0.
\endaligned\right.
$$
where $\hat{L} = \Delta_L^{h_0} + 2m - 2 \mathcal{P}(h_0)$.
Since $\mathcal{P}(h_0)\in x^4C^{4+\alpha }_e(X)\subset xC^{2+\alpha}(\bar X)$,
 $\hat{L} $ is uniformly degenerate elliptic. 
 We may apply  Lemma \ref{thm 3.2 in Ba11} to $\partial_t - \hat{L}$. 
 \\

Let us define the mapping as follows: given $u \in Z_{\epsilon, T} $,  
from the discussions in the previous subsection, we know
\begin{equation*}
 \mathcal{Q}(u)  + \hat{\mathcal{Q}}(u)+\mathcal{E}  + \hat{\mathcal{E}} +
 \hat{\mathcal{M}}(u)\in x^2C^{\alpha, \frac \alpha 2}_e(X_T).
\end{equation*}
 We may solve the linear equations by Lemma \ref{thm 3.2 in Ba11}
\begin{equation} \label{eq extra-vanishing}
\left\{\aligned
(\partial_t - \hat L)v & = \mathcal{Q}(u)+ \hat{\mathcal{Q}}(u) + \mathcal{E}  
+ \hat{\mathcal{E}}  +\hat{\mathcal{M}}(u) , \\
v(\cdot, 0) & = 0.
\endaligned\right.
\end{equation}
By the parabolic Schauder estimates in Lemma \ref{thm 3.2 in Ba11} we may define a map
$$
\Psi: Z_{\epsilon,T} \subseteq x^2 C^{2+\alpha, \frac{\alpha}2}_e(X_T)
\to x^2 C^{2+\alpha, \frac {\alpha}2}_e(X_T), \quad \Psi (u)=v_u =v.
$$
For appropriate choices of $\epsilon$ and $T$, we claim that the map
\begin{equation}\label{mapping}
\Psi: Z_{\epsilon,T} \to Z_{\epsilon,T},
\end{equation}
and that $\Psi$ is contractive.
In the following lemmas, we verify the above  claim and then apply the contraction mapping theorem 
to prove the short time existence 
for DCRF (\ref{eq dcrf}).

\begin{lemma} \label{lem Ba 4.5}
The mapping $\Psi$ maps $Z_{\epsilon,T}$ into $Z_{\epsilon,T}$ when $\epsilon$ and $T$ are small enough.
\end{lemma}

\begin{proof} We may write $v=v_u = v_{u1} + v_{u2} + v_{u3} = v_1 + v_2 + v_3$, where
$$
\left\{ \aligned (\partial_t - \hat{L}) v_1 & = \mathcal{Q}(u) + \hat{\mathcal{Q}}(u), \\
v_1(\cdot, 0) & = 0,
\endaligned\right.
$$
$$
\left\{\aligned (\partial_t - \hat{L}) v_2 & =\mathcal{E} + \hat{\mathcal{E}},  \\
v_2(\cdot, 0) & = 0,
\endaligned\right.
$$
and
$$
\left\{\aligned (\partial_t - \hat{L}) v_3 & =\hat{\mathcal{M}}(u), \\
v_3(\cdot, 0) & = 0.
\endaligned\right.
$$

Due to the quadratic nature of the operators $\mathcal{Q}$ and $\tilde{\mathcal{Q}}$, 
as in the proof of \cite[Lemma 4.5]{Ba11}, we have
$$
\| v_1 \|_{x^2 C^{2+\alpha, \frac {2+\alpha}2}_e (X_T )} \leq \frac 13 \epsilon,
$$  
where  $T$ is arbitrarily given and small and $\epsilon$ is small enough. 

For $v_2$, by Lemma \ref{lem funct prop Hat E(u)}  and \ref{thm 3.2 in Ba11}
we have 
\[
\| v_2 \|_{x^2 C^{4+ \alpha, \frac{4+ \alpha }{2}}_e(X_T)}  \leq C \|\operatorname{Rc}_{h_0} 
 + m h_0 \|_{x^2 C^{2+ \alpha}_e(X)} ,
\]
this implies that 
\begin{equation}\label{eq L hat v 2 est}
\|\hat{L} v_2 \|_{x^2 C^{2+ \alpha, \frac{ \alpha }{2}}_e(X_T)}  \leq C \|\operatorname{Rc}_{h_0} 
 + m h_0 \|_{x^2 C^{2+ \alpha}_e(X)}.
\end{equation}
As in the proof of \cite[Lemma 4.5]{Ba11}, we may write 
$$
v_2 (\cdot, t) = \int_0^t (\mathcal{E} + \hat{\mathcal{E}} + \hat L v_2)ds,
$$
by (\ref{eq L hat v 2 est}) and Lemma \ref{lem funct prop Hat E(u)} we get 
$$
\|v_2\|_{x^2 C^{2+\alpha, \frac \alpha 2}(X_T)} \leq C T^{1-\frac \alpha 2}
 \|\operatorname{Rc}_{h_0} + m h_0\|_{x^2 C^{2+\alpha}_e(X)}
\leq \frac 13 \epsilon
$$
when $T$ is chosen sufficiently small.

For $v_3$, the argument is similar to that of $v_2$. By Lemma \ref{lem funct prop hat Q(u)}(i) 
and \ref{thm 3.2 in Ba11},  we have
\[
\|v_3 \|_{x^2 C^{4+ \alpha,  \frac{2+\alpha}{2}}_e(X_T)}  \leq C \|u \|_{x^2 C^{2+ \alpha, 
\frac{\alpha }{2} }_e(X_T)},
\]
hence
\begin{equation}\label{eq L hat v 3 est}
\|\hat{L} v_3 \|_{x^2 C^{2+ \alpha,  \frac{\alpha }{2} }_e(X_T)}  \leq C \|u \|_{x^2 C^{2+ \alpha,
\frac{\alpha }{2} }_e(X_T)}.
\end{equation}
We write 
$$
v_3(\cdot, t) = \int_0^t  (\hat{\mathcal{M}}(u) + \hat{L} v_3)ds,
$$ 
and get
$$
\|v_3\|_{x^2 C^{2+\alpha, \frac \alpha 2}_e(X_T)} \leq C T^{1 - \frac 
\alpha 2}\|u\|_{x^2 C^{2+\alpha, \frac{\alpha }{2}}_e(X_T)} \leq \frac 13 \epsilon
$$
when $T$ is chosen sufficiently  small. Thus the proof is complete.
\end{proof}

%%%%%%%%
\begin{lemma} \label{lem Ba 4.6}  
The map $\Psi: Z_{\epsilon, T} \to Z_{\epsilon, T}$ is contractive when $\epsilon$ and $T$ are small.
\end{lemma}

\begin{proof} Adopt the notations used in the proof of Lemma \ref{lem Ba 4.5}, 
for $u, \tilde{u} \in Z_{\epsilon, T}$ it is easy to see that $v_{u2} - v_{\tilde{u}2} =0$.

$v_{u1} - v_{\tilde{u}1}$ satisfies
$$
\left\{ \aligned 
& (\partial_t - \hat{L}) (v_{u1} - v_{\tilde{u}1}) =  \mathcal{Q}(u ) - \mathcal{Q}( \tilde{u}) 
 + \hat{\mathcal{Q}}(u )    - \hat{\mathcal{Q}}( \tilde{u}) ,  \\
& (v_{u1} - v_{\tilde{u}1})(\cdot, 0) = 0.
\endaligned\right.
$$
As in the proof of \cite[Lemma 4.6]{Ba11} and in the light of Lemma \ref{Lem 4etc in Ba11} 
and \ref{lem funct prop hat Q(u)}(iii), we have
\[
\|v_{u1} - v_{\tilde{u}1 } \|_{x^2 C^{2+\alpha, \frac \alpha 2}_e(X_T)} \leq \frac 13
 \|u - \tilde{u}  \|_{x^2 C^{2+\alpha, \frac \alpha 2}_e(X_T)}.
\]

Note that $v_{u3} - v_{\tilde{u}3}$ satisfies
$$
\left\{ \aligned 
& (\partial_t - \hat{L}) (v_{u3} - v_{\tilde{u}3}) = \hat{\mathcal{M}}(u )
    - \hat{\mathcal{M}}( \tilde{u}),   \\
& (v_{u3} - v_{\tilde{u}3})(\cdot, 0) = 0.
\endaligned\right.
$$
Since $$
\hat{\mathcal{M}}(u ) - \hat{\mathcal{M}}( \tilde{u}) = -2 (\mathcal{P}(h_0 + u)
 - \mathcal{P}(h_0+ \tilde{u} ))h_0,
$$
using the estimate in Lemma \ref{lem Laplace invertible weighted timed}(ii) we can adopt the
argument for the estimate of $v_3$ in the proof of  Lemma \ref{lem Ba 4.5} to get
$$
\|v_{u3} - v_{\tilde{u}3 } \|_{x^2 C^{2+\alpha, \frac \alpha 2}_e(X_T)} \leq 
C T^{1- \frac{\alpha}{2}}  \|u - \tilde{u}  \|_{x^2 C^{2+\alpha, \frac{\alpha}{2}}_e(X_T)}
\leq \frac 13
 \|u - \tilde{u}  \|_{x^2 C^{2+\alpha, \frac \alpha 2}_e(X_T)},
$$
when $T$ is chosen sufficiently  small.

It follows that map $\Psi$ is  contractive when  $\epsilon$ and $T$ are small enough.  
 \end{proof} 

Now let us summarize and state a short time existence theorem for DCRF.

\begin{theorem}\label{DCRF-short-time} 
Suppose that $(X^{m+1}, h_0)$ is $C^{4+\alpha}$ AH with 
constant scalar curvature $-m(m+1)$
and that $x$ is a geodesic defining function. Assume that $\operatorname{Rc}_{h_0} + mh_0$
 $\in x^2 C^{2+\alpha}_e(X)$. Then, for some small $T$, DCRF
\eqref{eq dcrf} has a solution $h(t) = h_0 + v(\cdot, t)$ such that $h(t)$ is $C^{1+\alpha}$
 AH with constant scalar curvature $-m(m+1)$ and $v \in 
x^2C^{2+\alpha, \frac {2+\alpha}2}_e(X_T)$.
\end{theorem}

\proof The only thing we need to point out is that, for a fixed point $v=u \in x^2C^{2+\alpha, 
\frac \alpha 2}_e(X_T)$, it is automatic that
$v \in x^2 C^{2+\alpha, \frac {2+\alpha}2}_e(X_T)$ due to the equation 
(\ref{eq extra-vanishing}).
\endproof
 
%%%%%%%%%%%%%%%%%%%%%%%%%%%%%%%%%%%%%%%%%%%%%%%%
\subsection{Short time existence for CRF on AH manifolds}
\label{subsec exist CRF AH} To construct CRF from DCRF, 
we considers the family of diffeomorphisms $\varphi (t)$
generated by the vector field $W(t)$:
$$
\frac d{dt} \varphi (t) = W(t) \quad \text{ and } \varphi(0) = \operatorname{id},
$$
where $W(t)$ is defined in \eqref{vector field}. Note that $W \in x C^\alpha
 (\overline{X_T})\bigcap x^2 C^{1+\alpha, \frac {1+\alpha}2}_e(X_T)$.
Let $h(t)$ be the solution of DCRF from Theorem \ref{DCRF-short-time}. Then $g (t) = 
(\varphi(t))^* h(t) = g_0 + u(\cdot, t)$ solves the CRF for short time and  
$u \in x^2 C^{2+\alpha, \frac {2+\alpha}2}_e(X_T)$ (cf. 
 \cite[Lemma 3.1]{LQZ}, for instance). Moreover, $g(t)$ is $C^{1+\alpha}$ AH with 
constant scalar curvature $-m(m+1)$.

%%%%%%%%%%%%%%%%%%%%%%%%%%%%%%%%%%%%%%%%%%%%%%%%

\section{Shi's curvature derivative estimates for CRF}
\label{sect Shi's derivative estimate under crf}

In this section we will give a proof of Theorem \ref{thm local curvature derivative est} and  Shi's 
estimates of high order derivative of curvature tensor for CRF (Theorem \ref{thm shi high order estimate}).
The following lemma will be used in the proof which is the CRF analog of
Lemma 14.3 in \cite{CC2} for Ricci flow.
The lemma can be proved by a straightforward modification of the Lemma 14.3 where
the comparison of Christoffel symbol $\Gamma_{ij}^k(g(t))$ with $\Gamma_{ij}^k(g(0))$
requires the bound of $\left \vert \nabla p \right \vert$.

\begin{lemma} \label{lem cutoff func eta}
Let $\left( M^n,  g\left(  t\right), p(t)  \right)  $, $t \in \lbrack0,T]$,
be a solution to CRF (\ref{eq crf g p local riem}) with $R_{g(0)}=2nc$.
Assume that closed ball $\bar{B}_{g\left(  0\right)  }\left( O, r\right)  \subset M$
is compact and that

\[
\left\vert \operatorname*{Rm}\right\vert \leq K, \quad \max_{i=0,1} \left \vert \nabla^i p \right
\vert \leq \tilde{K}  \quad \text{on }B_{g\left( 0\right)  }\left(O,r\right)
\times\lbrack0,T_*],
\]
where $T_* \leq \min \{T, \alpha/K \}$ for some positive constants $\alpha, K, \tilde{K}$.
Note that $|c| \leq n(n-1)K$.
Let
\[
\Theta\left(  x,t\right)
= tK^{2}\left\vert \nabla\operatorname{Rc}\left(  x,t\right)
\right\vert ^{2}.
\]
Then there exist constants $C_1 =C_1\left(  \alpha,n, \tilde{K}, |c| \right)$,
$C_2=C_2\left(  \alpha, n, \sqrt{K}r, \tilde{K}, |c| \right) $, and a cutoff function
$\eta:M \rightarrow\left[  0,1\right]  $ with support in $B_{g\left(
0\right)  }\left( O,r\right)$ such that for $(x,t) \in B_{g(0)}\left(O,r\right)  \times
\lbrack0, T_*]$ we have
\begin{align}
& \eta =1  \quad \text{ on } B_{g\left(  0\right)}\left( O,r/2\right)  \notag   \\
& \left\vert \nabla\eta\left(  x\right)  \right\vert _{g\left(  t\right)  }^{2}
  \leq\frac{C_1   }{r^{2}}\eta\left(  x\right)
,\label{grad eta estima0}\\
 &   -\Delta_{g\left(  t\right)  }\eta\left(  x \right)  \leq\frac{C_2  }{r^{2}}+\frac{C_1}
 {K^{3/2}r}\sup_{s\in\lbrack0,t]}\left(  \eta\Theta\right)  ^{1/2}\left(
x,s\right)  . \label{lapl eta estim00}
\end{align}
\end{lemma}

\vskip .2cm
\noindent {\it Proof of Theorem \ref{thm local curvature derivative est}}. \
We will use the formula of $\partial_t \operatorname{Rm}$ in \cite[p.417]{LQZ} 
 to compute 
\begin{equation*}
\partial_t \left \vert \operatorname{Rm} \right \vert^2
=\frac{\partial}{\partial t} \left  (g^{ri}g^{sj}g^{pk}g^{ql} R_{rspq}R_{ijkl} \right ).
\end{equation*}
Note that the terms, which contain factor $\operatorname{Rc}-2cg$ and arise
when we differentiate $g^{-1}$,
cancel the corresponding terms which arise when we differentiate $R_{ijkl}$.
We get the evolution equation for the norm of the curvature (compare \cite[p.225]{CK})
\begin{align}
\left ( \partial_t -\Delta  \right )\left \vert \operatorname{Rm} \right \vert^2
\leq &  - 2 \left  \vert \nabla \operatorname{Rm} \right \vert^2
+16 \left \vert \operatorname{Rm} \right \vert ^3
 +4(|p| + 2|c| ) \left \vert \operatorname{Rm} \right \vert^2 \label{eq evol nabla Rm}  \\
&   + 8 \left \vert \operatorname{Rm}
 \right \vert \left \vert \nabla^2 p \right \vert. \notag
\end{align}
Actually $\left \vert \operatorname{Rm}
 \right \vert \left \vert \nabla^2 p \right \vert$ term can be written as
 $\left \vert \operatorname{Rc} \right \vert \left \vert \nabla^2 p \right \vert$.
 \\

It follows by a standard computation using the formula of $\partial_t \operatorname{Rm}$ in 
\cite[p.417]{LQZ} 
that the covariant derivative $\nabla \operatorname{Rm}$  satisfies (compare \cite[p.227]{CK})
\begin{align*}
 \left (  \partial_t - \Delta \right )\nabla \operatorname{Rm}
= & 40 \operatorname{Rm} * \nabla \operatorname{Rm} + 8 (\operatorname{Rc}- 2cg) * \nabla
\operatorname{Rm} \\
&  -2(p-2c) \nabla \operatorname{Rm}+ 5 \nabla p * \operatorname{Rm}
 + 4 g* \nabla^3 p. 
\end{align*}
Here if $A$ and $B$ are tensors, $A*B$ means some contraction of the tensor product $A \otimes B$.
If $k$ are natural number, then $k A*B$ denotes a tensor consisting of $k$ terms of $A* B$.
Then it follows that
\begin{align}
& \left (\partial_t  - \Delta \right ) \left \vert \nabla \operatorname{Rm}
\right \vert^2  \label{eq nabla rm norm evol} \\
= &  - 2 \left  \vert \nabla^2  \operatorname{Rm} \right \vert^2
+ 90 \operatorname{Rm}* \nabla \operatorname{Rm} * \nabla \operatorname{Rm}
+ 16(\operatorname{Rc}-2cg) * \nabla \operatorname{Rm} * \nabla \operatorname{Rm}  \notag  \\
& + 14( p-2c) \nabla  \operatorname{Rm} * \nabla \operatorname{Rm}+ 10 \nabla p *  \operatorname{Rm}
* \nabla  \operatorname{Rm} + 8 \nabla  \operatorname{Rm} * \nabla^3 p . \notag
\end{align}
Actually term $\nabla  \operatorname{Rm} * \nabla^3  p$  can be written as
$\nabla  \operatorname{Rc} * \nabla^3 p$.
\\

Below $C(n)$ is a constant depending only on dimension $n$.
Using (\ref{eq evol nabla Rm}) and (\ref{eq nabla rm norm evol}) we compute
\begin{align*}
& \left (\partial_t  - \Delta \right )  \left ( \left(  16K^{2}
+\left\vert \operatorname*{Rm}\right\vert^{2}\right)  \left\vert \nabla
\operatorname*{Rm}\right\vert ^{2} \right ) \\
= &  \left ( \partial_t  - \Delta \right ) \left \vert \operatorname*{Rm} \right
\vert^{2} \cdot \left \vert \nabla \operatorname*{Rm} \right \vert ^{2}
+ \left(  16K^{2} +\left\vert \operatorname*{Rm}\right\vert^{2}\right) \cdot \left (
 \frac{\partial}{\partial t} - \Delta \right )   \left \vert \nabla \operatorname*{Rm}
\right \vert ^{2}   \\
& -2 \nabla \left\vert \operatorname*{Rm}\right\vert^{2} \cdot \nabla  \left \vert \nabla
\operatorname*{Rm}\right\vert ^{2}  \\
\leq & -2 \left  \vert \nabla \operatorname{Rm} \right \vert^4 + 16 \left  \vert
\operatorname{Rm} \right \vert^3 \left \vert \nabla \operatorname*{Rm} \right \vert ^{2}
-2  \left (  16K^{2} +\left\vert \operatorname*{Rm}\right\vert^{2}\right) \left  \vert \nabla^2 
 \operatorname{Rm} \right \vert^2 \\
& + C(n) \left (  16K^{2} +\left\vert \operatorname*{Rm}\right\vert^{2}\right) \left\vert
 \operatorname*{Rm}\right\vert \left\vert \nabla \operatorname*{Rm}\right\vert^{2}
+ 8 \left \vert  \operatorname*{Rm} \right \vert  \left\vert  \nabla \operatorname*{Rm}
\right\vert^{2}  \left \vert \nabla^2 \operatorname*{Rm}\right\vert \\
&+ C(n) |c| \left ( 16K^2 + \left \vert \operatorname*{Rm}\right\vert^2 \right )
\left \vert \nabla \operatorname*{Rm}
\right\vert^2  + C(n) |p| \left ( 16K^2 + \left \vert \operatorname*{Rm}\right \vert^2 \right )
\left \vert \nabla \operatorname*{Rm} \right \vert^2 \\
&+ C(n) \left \vert \nabla p \right \vert \left ( 16K^2 + \left \vert \operatorname*{Rm}\right
\vert^2 \right ) \left \vert \operatorname*{Rm} \right \vert \left \vert \nabla \operatorname*{Rm}
\right \vert +C(n) \left \vert \nabla^2 p \right \vert  \left \vert  \operatorname*{Rm}
\right \vert \left \vert \nabla \operatorname*{Rm} \right \vert^2 \\
& + C(n) \left \vert \nabla^3 p \right \vert \left ( 16K^2 + \left \vert \operatorname*{Rm}\right
\vert^2 \right ) \left \vert \nabla \operatorname*{Rm} \right \vert.
\end{align*}
Using the assumption $\left \vert \operatorname*{Rm} \right \vert \leq K$ with $K \geq 1$ and
 $\max_{i=0,1,2,3}  | \nabla^i p| \leq \tilde{K} $ on $ \bar{B}_{g
\left(  0\right) }\left( O,r\right) \times \lbrack0,T_*]$
we get that on  $ \bar{B}_{g \left(  0\right) }\left( O,r\right) \times \lbrack0,T_*]$
\begin{align*}
& \left (\partial_t  - \Delta \right )  \left ( \left(  16K^{2}
+\left\vert \operatorname*{Rm}\right\vert^{2}\right)  \left\vert \nabla
\operatorname*{Rm}\right\vert ^{2} \right ) \\
\leq &  -2 \left \vert \nabla \operatorname{Rm} \right \vert^4 + C(n, \tilde{K}, |c|) K^3
\left \vert \nabla \operatorname{Rm} \right \vert^2 -32 K^2 \left \vert \nabla^2
\operatorname{Rm} \right \vert^2 + C(n, \tilde{K}) K^6  \\
& + 8 K \left \vert \nabla \operatorname{Rm} \right \vert^2
\left \vert \nabla^2 \operatorname{Rm} \right \vert.
\end{align*}
Since
\begin{align*}
&-\frac{1}{2} \left \vert \nabla \operatorname{Rm} \right \vert^4 +
8 K \left \vert \nabla \operatorname{Rm} \right \vert^2
\left \vert \nabla^2 \operatorname{Rm} \right \vert -32 K^2 \left \vert \nabla^2
\operatorname{Rm} \right \vert^2 \leq 0, \\
& -\frac{1}{2} \left \vert \nabla \operatorname{Rm} \right \vert^4 + C(n, \tilde{K}, |c|) K^3
\left \vert \nabla \operatorname{Rm} \right \vert^2  \leq \tilde{C}(n, \tilde{K}, |c|) K^6,
\end{align*}
we have established
\begin{align}
\left ( \partial_t  - \Delta \right )  \left ( \left(  16K^{2}
+\left\vert \operatorname*{Rm}\right\vert^{2}\right)  \left\vert \nabla
\operatorname*{Rm}\right\vert ^{2} \right ) \leq -\left \vert \nabla \operatorname{Rm} \right \vert^4
+ C(n, \tilde{K}, |c|) K^6. \label{eq evol of Bernstein est quantity}
\end{align}

Inequality (\ref{eq evol of Bernstein est quantity}) is of the same form as the inequality on
the top of \cite[p.238]{CC2}. The remaining proof of estimate (\ref{eq est local first order Shi})
can be finished by the same argument as the proof given in \cite[pp.238--239]{CC2} for the Ricci
flow. Roughly speaking this is done in two steps. First we localize the inequality
(\ref{eq evol of Bernstein est quantity}) by multiplying $\left(  16K^{2}
+\left\vert \operatorname*{Rm}\right\vert^{2} \right ) \left\vert \nabla
\operatorname*{Rm}\right\vert ^{2} $ by $t \eta$
where $\eta$ is the cutoff function given in Lemma \ref{lem cutoff func eta}.
Then we apply the parabolic maximum principle to the resulting inequality of the localized quantity.
We omit the details.
\hfill $\square$

\vskip .3cm
Now we turn to Shi's high order derivative estimates. 
First we compute the evolution equation of high derivatives of the curvature tensor,
here we follow closely the calculation for Ricci flow (see \cite[p.228]{CK},
for example).
\begin{align*}
\partial_t  \nabla^k \operatorname{Rm} = &  \nabla^k \left (\partial_t 
\operatorname{Rm} \right ) + \sum_{j=0}^{k-1} \nabla^j \left ( \nabla(
\operatorname{Rc} + (p-2c)) * \nabla^{k-j-1}  \operatorname{Rm} \right )  \\
=&  \nabla^k ( \Delta \operatorname{Rm} + \operatorname{Rm} * \operatorname{Rm} +
(\operatorname{Rc}-2cg)* \operatorname{Rm} -2(p-2c)\operatorname{Rm} +g*\nabla \nabla p )\\
& + \sum_{j=0}^{k-1} \nabla^j \left ( \nabla \operatorname{Rc} * \nabla^{k-j-1}
\operatorname{Rm} \right ) + \sum_{j=0}^{k-1} \nabla^j \left ( \nabla p *
\nabla^{k-j-1}  \operatorname{Rm} \right ) \\
 = & \nabla^k \Delta \operatorname{Rm}  + \sum_{j=0}^k \nabla^j \operatorname{Rm} *
  \nabla^{k-j} \operatorname{Rm} + c g* \nabla^k \operatorname{Rm} \\
 &  + \sum_{j=0}^k  \nabla^j p* \nabla^{k-j} \operatorname{Rm}+ g* \nabla^{k+2} p .
\end{align*}
Since for any tensor $A$ we have that the commutator
\[
[\nabla^k, \Delta] A = \sum_{j=0}^k \nabla^j \operatorname{Rm} * \nabla^{k-j} A,
\]
we conclude

\begin{align}
& \left ( \partial_t  - \Delta \right )  \nabla^k \operatorname{Rm}
\label{eq nabla k curv evol}\\
=& \sum_{j=0}^k \nabla^j \operatorname{Rm} * \nabla^{k-j} \operatorname{Rm} + c
g* \nabla^k \operatorname{Rm} + \sum_{j=0}^k \nabla^j p* \nabla^{k-j} \operatorname{Rm}
 + g* \nabla^{k+2} p . \notag
\end{align}

Note that
\[
\partial_t  \left \vert \nabla^k \operatorname{Rm} \right \vert^2  =
2 \left \langle \partial_t  \left ( \nabla^k \operatorname{Rm} \right ),
\nabla^k \operatorname{Rm} \right \rangle
+ (\operatorname{Rc} +(p-2c))*\nabla^k \operatorname{Rm}*\nabla^k \operatorname{Rm}.
\]
Using the following equality for tensors
\[
2 \langle \Delta A, A \rangle = \Delta |A|^2 - 2|\nabla A|^2,
\]
we get
\begin{align*}
& \left (\partial_t  - \Delta\right ) \left \vert \nabla^k \operatorname{Rm}
\right \vert^2 \\
= & - 2 \left  \vert \nabla^{k+1} \operatorname{Rm}
\right \vert^2 + \sum_{j=0}^k \nabla^j \operatorname{Rm}* \nabla^{k-j} \operatorname{Rm} 
* \nabla^k \operatorname{Rm} + c \nabla^k \operatorname{Rm} * \nabla^k \operatorname{Rm} \\
& +  \sum_{j=0}^k \nabla^j p
* \nabla^{k-j}  \operatorname{Rm} * \nabla^k  \operatorname{Rm}
+  \nabla^{k+2} p * \nabla^k \operatorname{Rm}.
\end{align*}

Using the above evolution equation we can prove the following local estimate of the high order 
derivative of curvature tensors for CRF by applying maximum principle to the evolution 
inequality of  the localized quantity
\[
\eta \left ( C +t^n \vert \nabla^n \operatorname{Rm} \vert^2 \right ) t^{n+1}
\vert \nabla^{n+1} \operatorname{Rm} \vert^2.
\]
We omit the detail of the proof.

\begin{theorem} \label{thm shi high order estimate}
(i) There exists a constant $C_1 =C_1 (
\alpha,n, m, r,K ,\tilde{K} )$ depending only on $\alpha, n, m$, $r$, $K$,
and $\tilde{K}$, such that the following property holds.
Let $(M^{n}, g(t)$, $p(t))$,  $t\in\left[
0,T \right] $, be a solution to CRF (\ref{eq crf g p local riem}) with $R_{g(0)}=2nc$.
Assume that closed ball $\bar{B}_{g\left(  0\right)
}\left( O,r\right) \subset M $ is compact and that
\begin{align}
&\left\vert \operatorname*{Rm}  \right\vert \leq K \text{ on  } \bar{B}_{g\left(
 0\right) }\left( O,r\right) \times \lbrack0, T_*]
\label{eq curv bdd assump 3}  \\
& \max_{i=0,1,\cdots, m+2}  | \nabla^i p| \leq \tilde{K} \text{ on  }
\bar{B}_{g\left(  0\right)
}\left( O,r\right) \times \lbrack0, T_*] \label{eq p deriv bdd assump 3}
\end{align}
where $ T_* \leq \min \{T, \alpha/K \}$, then we have
\begin{equation}
\left\vert \nabla^m \operatorname*{Rm}\left(  x,t\right)  \right\vert \leq
\frac{C_1 }{t^{m/2}} \label{eq est local high order Shi}
\end{equation}
for all $\left(  x,t\right)  \in B_{g\left(  0\right)  }\left(
O, r/2\right)  \times(0,T_*]$.

(ii) If  $(M^n,g(t), p(t)), \, t \in [0,T]$, in (i) is a complete solution to CRF. Suppose assumption
(\ref{eq curv bdd assump 3}) and (\ref{eq p deriv bdd assump 3}) holds on $M \times [0,T_*]$,
then there is a constant $C_2 =C_2(\alpha, n, m,\tilde{K}, |c|)$ depending only on $\alpha, n, m,
\tilde{K}$, and $|c|$, such that
\begin{equation}
\left\vert \nabla^m \operatorname*{Rm}\left(  x,t\right)  \right\vert \leq
\frac{C_2 K}{t^{m/2}} \label{eq est high order Shi 3}
\end{equation}
for all $\left(  x,t\right)  \in M \times(0,T_*]$.
\end{theorem}

\begin{remark}
(i) In Theorem \ref{thm shi high order estimate}(i)
 if we further assume $\left\vert \nabla^k \operatorname*{Rm}\left(  x, 0 \right)
\right\vert  \leq K$ for $k=1, \cdots, l$ and $x\in \bar{B}_{g\left(  0\right) }\left( O,
r\right)$, then we have $\left\vert \nabla^m
\operatorname*{Rm}\left(  x,t\right)  \right\vert \leq \frac{C_3 }{t^{(m-l)/2}}$ 
for all $\left( x,t\right)  \in B_{g\left(  0\right)  }\left(
O,r/2\right)  \times(0,T_*]$. Here $C_3=C_3(\alpha,n, m,r, K,
\tilde{K})$ is a constant. A similar generalization of Theorem
\ref{thm shi high order estimate}(ii) also holds. This is the analog of the so-called 
modified Shi's local derivative estimates in Ricci flow (\cite[Theorem 14.16]{CC2}).

(ii) Using Theorem \ref{thm shi high order estimate} we can prove a compactness 
theorem for CRF.  Let $\left\{  ({M}_{k}^{n},g_{k}(t),p_{k}(t), O_k)\right\}$,
$t\in  ( -\alpha, \beta)$ with $\alpha, \beta >0$, be a sequence of pointed complete
solutions of CRF with constant scalar curvature.
Assume that for some constants $m \in \mathbb{N}$, $K$, and $\tilde{K}$ 
\begin{align}
& \vert \operatorname{Rm}_{g_k} \vert \leq K \quad \text{ on }  M_k \times (-\alpha, \beta),
\label{eq curv bdd assump 4}  \\
& \max_{i=0,1,\cdots, m+2} | \nabla^i p| \leq \tilde{K} \quad \text{ on
}  M_k \times (-\alpha, \beta). \label{eq p deriv bdd assump 4}
\end{align}
Further assume that the injectivity radius $\operatorname{inj}_{g_{k}
(0)}(O_{k}) \geq \delta$ for some $\delta>0$.
With the aid of Theorem \ref{thm shi high order estimate}(ii) we may apply the
Cheeger--Gromov compactness theorem to the sequence and conclude the following.
There exists a subsequence of $\left\{  (M_{k},g_{k}(t),p_{k}(t), O_k)\right\} $,
$t \in ( -\alpha, \beta)$, which converges in
the pointed $C^{m +1}$-Cheeger--Gromov topology to a pointed complete
solution of CRF $(M_{\infty}^{n},g_{\infty}
(t),p_{\infty}(t), O_\infty )$, $t\in ( -\alpha, \beta)$.
\end{remark}

%%%%%%%%%%%%%%%%%%%%%%%%%%%%%%%%%%%%%%%%%%%%%%%%%%

\bibliographystyle{natbib}

\end{document}